\documentclass[11pt]{article}
\usepackage{amssymb}
\usepackage{amsmath}
\usepackage{amsthm}
\usepackage{amsfonts}
\usepackage[latin5]{inputenc}
\usepackage{color}

\begin{document}
\numberwithin{equation}{section}
\newcounter{thmcounter}
\newcounter{Remarkcounter}
\newcounter{Defcounter}
\newcounter{Assucounter}
\numberwithin{thmcounter}{section}
\numberwithin{Assucounter}{thmcounter}
\newcommand*{\QEDA}{\hfill\ensuremath{\blacksquare}}

\renewcommand{\labelenumi}{(\theenumi)}

\newtheorem{Assu}[Assucounter]{Assumption}
\newtheorem{Prop}[thmcounter]{Proposition}
\newtheorem{Corol}[thmcounter]{Corollary}
\newtheorem{theorem}[thmcounter]{Theorem}
\newtheorem{Lemma}[thmcounter]{Lemma}
\theoremstyle{definition}
\newtheorem{Def}[Defcounter]{Definition}
\theoremstyle{remark}
\newtheorem{Remark}[Remarkcounter]{Remark}
\newtheorem{Example}[Remarkcounter]{Example}

\newcounter{defcounter}
\setcounter{defcounter}{0}
\newcounter{defcounteri}
\setcounter{defcounteri}{0}
\newenvironment{BihEq}{%
\addtocounter{equation}{-1}
\refstepcounter{defcounter}
\renewcommand\theequation{BH\thedefcounter}
\begin{equation}}
{\end{equation}}
\newenvironment{BicEq}{%
\addtocounter{equation}{-1}
\refstepcounter{defcounteri}
\renewcommand\theequation{BC\thedefcounteri}
\begin{equation}}
{\end{equation}}

\newcommand{\diag}{\mathrm{diag}\ }
\newcommand{\DIV}{\mathrm{div}}
\title{On biconservative hypersurfaces in 4-dimensional Riemannian space forms}
\author{ Nurettin Cenk Turgay \footnote{Istanbul Technical University, Faculty of Science and Letters, Department of  Mathematics, 34469 Maslak, Istanbul, Turkey, e-mail: turgayn@itu.edu.tr} and Abhitosh Upadhyay \footnote{ Harish Chandra Research Institute, Chhatnag Road, Jhunsi, Allahabad, 211019, India, e-mail: abhi.basti.ipu@gmail.com, abhitoshupadhyay@hri.res.in}}
\maketitle
\begin{abstract}
In this paper, we study biconservative hypersurfaces in $\mathbb S^{n}$ and $\mathbb H^{n}$. Further, we obtain complete explicit classification of biconservative hypersurfaces in $4$-dimensional Riemannian space form with exactly three distinct principal curvatures.

\end{abstract}

\section{Introduction}
Let $(M^{m},g)$ and $(N^{n},h)$ be some Riemannian manifolds. Then, the energy functional $E$ is defined by 
$$E(\psi) = \frac{1}{2}\int_M|d\psi|^{2}dv$$
for any smooth mapping $\psi: M\rightarrow N$, where $d\psi$ denotes the differential of $\psi$ and $dv$ stands for the volume element of $g$. A mapping  $\psi$ is said to be harmonic if it is a critical point of the energy functional $E$. It is well known that a harmonic mapping $\psi$ satisfy the Euler-Lagrange equation 
$$\tau(\psi) = 0,$$ where $\tau(\psi)=\mathrm{tr}\nabla d\psi $ is the tension field of $\psi$ (See for example, \cite{EellsLemaire1983}). 

In 1964, Eells and Sampson proposed an infinite dimensional Morse theory on the manifold of smooth maps between Riemannian manifolds whereas their results describe harmonic maps more rigorously \cite{EellsSampson1964}. Further, J. Eells and L. Lemaire proposed the problem to consider the $k$-harmonic maps in \cite{EellsLemaire1983}. A particular interest has the case $k = 2$. The bienergy functional is defined by
\begin{center}
$E_2(\psi) = \int_M|\tau_1(\psi)|^{2}dv$
\end{center} 
for a smooth mapping  $\psi: M\rightarrow N$. The study of bienergy plays a very important role not only in elasticity and hydrodynamics, but also it can be seen as the next stage where the theory of harmonic maps fail. For example, Eells and Wood showed in \cite{EellsWood} that in case of 2-torus $T^{2}$ and the 2-sphere $S^{2}$, there exists no harmonic map from $T^{2}$ to $S^{2}$, whatever the metrics chosen in the homotopy classes of Brower degrees $\pm 1$, but in case of biharmonicity, the situation is completely different.

Biharmonic maps are a natural generalization of harmonic maps. A map $\psi$ is called {\it biharmonic} if it is a critical point of the bi-energy functional $E_2$. 
 In \cite{Ji}, G.Y. Jiang studied the first and second variation formulas of $E_{2}$ for which critical points are called biharmonic maps.  The Euler-Lagrange equation associated with this bi-energy functional is $$\tau_2(\psi)=0,$$ where
$\tau_2(\psi)=\Delta\tau(\psi)-tr\tilde R(d\psi,\tau(\psi))d\psi $
is the bi-tension field and $\Delta$ is the rough Laplacian acting on the sections of $\psi^{-1}(TN)$.  A harmonic map is obviously a biharmonic map. Because of this reason, a non-harmonic biharmonic maps is said to be a proper biharmonic map. Note that one can easily construct a  proper biharmonic map, by choosing a third order polynomial mapping between Euclidean spaces, since, in this situation, the biharmonic operator is nothing but the Laplacian composed with itself.

 In the last decades, Biharmonic submanifolds has become a popular subject of research with many significant progresses made by geometers around the world. One of the fundamental problems in the study of biharmonic submanifolds is to classify such submanifolds in a model space. So far, most of the work done has been focused on classification of biharmonic submanifolds of space forms.

The {\em stress-energy} tensor, described by Hilbert \cite{Hi}, is a 
symmetric $2$-covariant tensor $S$ associated to a variational 
problem that is conservative at the critical points. Such tensor was 
employed by Baird-Eells \cite{BE} in the study of harmonic maps.
In this context, it is given by
$$S=\frac{1}{2}|d\psi|^2g -\psi^\ast h,$$
and it satisfies 
$$div S=-<\tau(\psi),d \psi>.$$
Therefore, $div S=0$ when $\psi$ is harmonic. The study of stress energy tensor, in the context of biharmonic maps,  was initiated by Jiang in \cite{Ji1} and afterwards developed by Loubeau, Montaldo and Onicuic in \cite{ELSMCO20141}. It is given by
\begin{eqnarray*}
S_2(X,Y) &=& \frac{1}{2}|\tau(\psi)|^2< X,Y> + 
<d\psi,\nabla\tau(\psi)>< X,Y> \\
&& -<X(\psi),\nabla_Y\tau(\psi)> - <Y(\psi),\nabla_X\tau(\psi)>,
\end{eqnarray*}
which satisfies
\begin{equation}\label{Biconservativestressenerytensor}
div S_2=-<\tau_2(\psi),d\psi>.
\end{equation}
This means that an isometric immersions with $div S_{2} = 0$ correspond to immersions with vanishing tangent part of the corresponding bitension field. If $\psi: M\rightarrow N$ is an isometric immersion, then equation (\ref{Biconservativestressenerytensor}) becomes

$$div S_2=-\tau_2(\psi)^{T}.$$
Now, we have the following:
\begin{Def}
Let $\psi:M\to N$ be an  isometric immersion  between two Riemannian 
manifolds. $\psi$  is called {\em biconservative} if its
stress-energy tensor $S_2$ is conservative, i.e., $\tau_2(\psi)^T=0$, where 
$\tau_2(\psi)$ is the bitension field of $\psi$.
\end{Def}
The class of biconservative submanifolds includes that of biharmonic submanifolds, which have been of large interest in the last decade \cite{CMOP, FetcuOniciucPinheiro, Fu2013MinkBih, YuFu2015, ELSMCO20141, SMCOAR, SMCOAR1, Sasahara}. It is well known that $\psi$ is biconservative if and only if 
\begin{BicEq}\label{BiconservativeEquationMostGeneral}
m\nabla\|H\|^2 + 4\mathrm{trace } S_{\nabla^\perp_{\cdot}H}(\cdot) + 
4\mathrm{trace }\big(\tilde R(\cdot,H)\cdot\big)^T=0,
\end{BicEq}where $H$, $\nabla^\perp$, $S$ are the mean curvature, the normal connection, the shape operator of $M$, respectively and  $\tilde R$ is the curvature tensor of $N$. 

In 1995, Hasanis and Vlachos initiated to study the biconservative hypersurfaces in the Euclidean space $\mathbb{R}^{n}$ and classified it in $\mathbb{R}^{3}$ and $\mathbb{R}^{4}$ \cite{HsanisField}. In that paper, authors called biconservative hypersurfaces as $\mathbb{H}$-hypersurfaces. In \cite{CMOP} and \cite{HsanisField}, the authors have classified proper biconservative surfaces in $\mathbb{R}^{3}$ and proved that they must be of surface of revolution. Further, Chen and Munteanu studied biconservative ideal hypersurfaces in Euclidean spaces $\mathbb{E}^{n} (n\geq 3)$ proving that they are either minimal or spherical hypercylinder \cite{chenMunteanu2013}. Recent results in the field of biconservative submanifolds were obtained, for example, in \cite{FetcuOniciucPinheiro, YuFu2015, YuFuTurgay, Sasahara,TurgayHHypersurface}. In \cite{TurgayHHypersurface}, first author studied biconservative hypersurfaces with diagonalizable shape operators in Euclidean spaces with exactly three distinct principal curvatures. Further, Yu Fu and Turgay obtained the complete classification of biconservative hypersurfaces in 4-dimensional minkowaski space with diagonalizable shape operators \cite{YuFuTurgay}. Furthermore, in \cite{AbhNCTJMAA} authors extended their study to biconservative hypersurfaces of $\mathbb{E}^{5}_{2}$ and obtained some classification results in this direction. Now, the natural question arises: Can we also classify all biconservative hypersurfaces in $\mathbb{S}^{n}$ or $\mathbb{H}^{n}$. In this paper, authors tried to classify all biconservative hypersurfaces in $\mathbb{S}^{4}$ and $\mathbb{H}^{4}$.

Our paper is organized as follows. In Section 1, we have presented a brief introduction of the previous work which has been done in this direction. In Section 2, we have collected the formulae and information which are useful
in our subsequent sections. In Section 3, we obtain some results for biconservative hypersurfaces in Riemannian space forms of arbitrary dimension. In particular, we get the form of position vector of a biconservative hypersurface in Section 3.1 without considering any restriction (See Theorem \ref{TheoremLocalClassForAnyBicHypr}) and further in Section 3.2, we focus on biconservative hypersurfaces with three distinct principal curvatures. Finally, in Section 4,  we study biconservative hypersurfaces in  $\mathbb{S}^{4}$  and $\mathbb{H}^{4}$ (See Theorems \ref{Classificationtheoremsphere} and \ref{Classificationtheoremhypr}).

The hypersurfaces which we are dealing are smooth and connected unless otherwise stated.

\section{Prelimineries}

Let $\mathbb E^m_s$ denote the semi-Euclidean $m$-space with the canonical 
Euclidean metric tensor of index $s$ given by  
$$
 g=\langle \cdot, \cdot \rangle=-\sum\limits_{i=1}^s dx_i^2+\sum\limits_{j=s+1}^m dx_j^2,
$$
where $(x_1, x_2, \hdots, x_m)$  is a rectangular coordinate system in $\mathbb E^m_s$. We put 
\begin{eqnarray} 
\mathbb S^{n}(r^2)&=&\{x\in\mathbb E^{n+1}: \langle x, x \rangle=r^{-2}\},\notag
\\  
\mathbb H^{n}(-r^2)&=&\{x\in\mathbb E^{n+1}_1: \langle x, x \rangle=-r^{-2}\}.\notag
\end{eqnarray}
These complete Riemannian manifolds, which have constant sectional curvatures, are called Riemannian space forms. We use the following notation
$$R^n(c)=\left\{\begin{array}{rcl} \mathbb S^{n}(c) &\mbox{if}& c>0\\ \mathbb H^{n} (c) &\mbox{if}& c<0 \end{array}\right.,\quad E(n,c)=\left\{\begin{array}{rcl} \mathbb E^{n+1} &\mbox{if}& c>0\\\mathbb E^{n+1}_{1} &\mbox{if}& c<0 \end{array}\right.$$
from which we see $R^n(c)\subset E(n,c)$. 

\subsection{Hypersurfaces in $R^{n+1}(c)$}
Consider an oriented hypersurface $M$ in $R^{n+1}(c)$ with unit normal vector field $N$. We denote Levi-Civita connections of $E(n+1,c)$, $R^{n+1}(c)$ and $M$ by $\hat{\nabla}$, $\widetilde{\nabla}$ and $\nabla$, respectively. Then, the Gauss and Weingarten formulas are given, respectively, by
\begin{eqnarray}
\label{MEtomGauss} \widetilde\nabla_X Y&=& \nabla_X Y + h(X,Y),\\
\label{MEtomWeingarten} \widetilde\nabla_X N&=& -S(X)
\end{eqnarray}
for all tangent vectors fields $X,\ Y\in M$, where $h$ and $S$ are the second fundamental form  and the shape operator of $M$, respectively. The Gauss and Codazzi equations are given, respectively, by
\begin{eqnarray}
\label{RnGaussEquation}\label{GaussEq} R(X,Y,Z,W)&=&c\Big(\langle Y,Z\rangle \langle X,W\rangle -\langle X,Z\rangle \langle Y,W\rangle \Big)+ \langle h(Y,Z),h(X,W)\rangle\\\nonumber&&
-\langle h(X,Z),h(Y,W)\rangle,\\
\label{MinkCodazzi} (\bar \nabla_X h )(Y,Z)&=&(\bar \nabla_Y h )(X,Z),
\end{eqnarray}
where  $R$ is the curvature tensor associated with connection $\nabla$ and  $\bar \nabla h$ is defined by
$$(\bar \nabla_X h)(Y,Z)=\nabla^\perp_X h(Y,Z)-h(\nabla_X Y,Z)-h(Y,\nabla_X Z).$$

Let $\{e_1,e_2,\hdots,e_n\}$ be a local orthornomal base field of the tangent bundle of $M$ consisting of  principal directions of $M$ with corresponding principal curvatures $k_1,\ k_2,\hdots,\ k_n$. Then, the second fundamental form of $M$ becomes
$$h(e_i,e_j)=\delta{ij}k_iN.$$
On the other hand, we denote the connection forms corresponding to  this frame field by $\omega_{ij}$, i.e., $\omega_{ij}(e_l)=\langle\nabla_{e_l}e_i,e_j\rangle$.  Note that we have $\omega_{ij}=-\omega_{ji}$. Thus, the Levi-Civita connection of $M$ satisfies
$$\nabla_{e_i}e_j=\sum\limits_k\omega_{jk}(e_i)e_k.$$
From the Codazzi equation \eqref{MinkCodazzi}, we have
\begin{subequations}\label{CodazziEq1All}
\begin{eqnarray}
\label{CodazziEq1a}e_i(k_j)=\omega_{ij}(e_j)(k_i-k_j),\\
\label{CodazziEq1b}\omega_{ij}(e_l)(k_i-k_j)=\omega_{il}(e_j)(k_i-k_l)
\end{eqnarray}
\end{subequations}
whenever $i,j,l$ are distinct.

 Let $\psi:M^n\rightarrow\mathbb R^{n+1}(c)$ be an isometric immersion and $i$ denote the canonical inclusion map. Then, we have
$$i:R^{n+1}(c)\rightarrow E(n+1,c) \hspace{.2 cm}and\hspace{.2 cm} x=i\circ \psi:M\rightarrow E(n+1,c).$$
\begin{Remark}\label{RemarkForSecondFundementalForms}
Put  $h$ and $h^*$ for  the second fundamental form of $\psi$ and $x$, respectively. It implies that
$$h^*(i_*X,i_*Y)=h(X,Y)-c\langle X,Y\rangle x,$$
if $X,Y$ are two vector field tangent to $M$ in $R^{n+1}(c)$. 
\end{Remark}

\section{Biconservative Hypersurfaces in $R^{n+1}(c)$}
In this section, we consider biconservative hypersurfaces in $R^{n+1}(c)$ for $c\in\{-1,+1\}$. The similar computation has been made for $\mathbb E^{n+1}$, $\mathbb E^{n+1}_1$ and $\mathbb E^{n+1}_2$ in some papers \cite{YuFuTurgay, TurgayHHypersurface, AbhNCTJMAA}.

Let $\psi:M\rightarrow\mathbb R^{n+1}(c)$ be isometric immersion, where $M$ is a hypersurface of $R^{n+1}(c)$. Then, by a direct computation using   \eqref{BiconservativeEquationMostGeneral}, we see that $\psi$ is biconservative if and only if
\begin{BicEq}\label{BiconservativeEquationinRn1}
S(\nabla H)+ \frac {nH}2\nabla H=0,
\end{BicEq}
where $S$ is the shape operator of $M$. Here, $M$ is called a biconservative hypersurface.

\begin{Remark}\label{FirstRemark}
We note that \eqref{BiconservativeEquationinRn1} is satisfied trivially if $H$ is constant. Therefore, we will locally assume that $\nabla H$ does not vanish. 
\end{Remark}
Let $M$ is a biconservative hypersurface and consider $e_1=\nabla H/|\nabla H|$. Therefore,  equation \eqref{BiconservativeEquationinRn1} implies that $k_1=-\frac {nH}2$. 
As $e_1$ is propotional to $\nabla H$, we have 
\begin{subequations}\label{BicHyperEq1}
\begin{equation}\label{BicHyperEq1a}
e_2(H)=\cdots=e_n(H)=0.
\end{equation}
Further, Remark \ref{FirstRemark} yields $e_1(H)\neq 0$ and locally we can suppose $H\neq0$. Therefore, by replacing $e_1$ with $-e_1$ and/or $N$ with $-N$ if necessary, we also assume
\begin{equation}\label{BicHyperEq1b}
e_1(H)>0\quad\mbox{ and }\quad H>0.
\end{equation}
\end{subequations}

\begin{Remark}
If the algebraic multiplicity of $k_1$ is more than 1, i.e., $k_1=k_A$ for some $A$, then  the Codazzi equation \eqref{CodazziEq1a} for $i=1,\ j=A$ gives $e_1(k_A)=0$, which contradicts to equation \eqref{BicHyperEq1}. Therefore, the function $k_1-k_A$ does not vanish for each $A$.
\end{Remark}
Since $\sum\limits_i k_i=nH$ and $k_1=-\frac {nH}2$, we have
\begin{equation}\label{BicEqMod2}
3k_1+k_2+\cdots+k_n=0.
\end{equation}
By considering  equation \eqref{BicHyperEq1a} and the Codazzi equation \eqref{CodazziEq1a}, we obtain
\begin{equation}\label{BicHyperEq2}
\omega_{1A}(e_1)=0,\quad i=A,\quad j=1, A>1.
\end{equation}
Further, taking into account $[e_A,e_B](k_1)=0$ and the Codazzi equation \eqref{CodazziEq1b}, we get
\begin{subequations}\label{BicHyperEq3ALL}
\begin{equation}\label{BicHyperEq3a}
\omega_{1A}(e_B)=0,\quad\mbox{whenever $A\neq B$, $A,B>1$}
\end{equation}
and
\begin{equation}\label{BicHyperEq3b}
\omega_{AB}(e_1)=0,\quad\mbox{ whenever $k_A\neq k_B$.}
\end{equation}
\end{subequations}
The Gauss equation \eqref{RnGaussEquation} for $X=Z=e_A,Y=W=e_1$ gives
\begin{equation}\label{GaussEq4}
e_1(\omega_{1A}(e_A))=-2c-k_1k_A-(\omega_{1A}(e_A))^2.
\end{equation}


\subsection{A local parametrization for biconservative hypersurfaces in $R^{n+1}(c)$}
The aim of this subsection is to obtain a local parametrization for biconservative hypersurfaces in $R^n(c)$.

We will use the following three lemmas in the next section. It is to note that in the proofs of these lemmas $\gamma=\gamma(s)$ denote an integral curve of $e_1$, i.e., $e_1\left|_{\gamma(s)}\right.=\gamma'(s)$, $H(s)=H\circ\gamma$, $N(s)=N\left|_{\gamma(s)}\right.$,  $e_A(s)=e_A\left|_{\gamma(s)}\right.$, $T(s)=\gamma'(s)$, $y=x\circ\gamma$ and $(\hat\nabla_T\zeta)\left|_{\gamma(s)}\right.=\zeta'(s)$ for any vector field $\zeta$ along $\gamma$, where $A=2,3,\hdots,n$.

\begin{Lemma}\label{Intcurves_e1Lem1}
Let $M$ be a  biconservative hypersurface in $R^{n+1}(c)$ and $e_1=\frac{\nabla H}{|\nabla H|}$, where $H$ is the mean curvature of $M$ and $c\in\{-1,1\}$. Then, any integral curve $\gamma$ of $e_1$ lies on 2-dimensional totally geodesic submanifold $R^2(c)$ of $R^{n+1}(c)$ and its curvature $\kappa_R$ in $R^2(c)$ is 
\begin{eqnarray}
\label{IntCurvEq1}\kappa_S&=&\frac{-nH}2.
\end{eqnarray}
\end{Lemma}

\begin{proof}
From equation \eqref{BicHyperEq2}, we have
\begin{align}\nonumber
\begin{split}
\widetilde\nabla_{T(s)}T(s)=-\frac{nH(s)}2N(s),\\
\widetilde\nabla_{T(s)}N(s)=\frac{nH(s)}2N(s).
\end{split}
\end{align}
Therefore, $\gamma$ lies on 2-dimensional totally geodesic submanifold of $R^n(c)$ with spherical curvature given in \eqref{IntCurvEq1}. 
\end{proof}


\begin{Lemma}\label{Intcurves_e1Lem2}
Let $M$ be a  biconservative hypersurface in $\mathbb S^{n+1}$ with $e_1=\frac{\nabla H}{|\nabla H|}$ and $\gamma$ is an integral curve of $e_1$ passing through $m\in M$, where $H$ is the mean curvature of $M$. Then, $\gamma$ lies on a $3$-plane  of $\mathbb E^{n+2}$ spanned by $e_1\left|_{m}\right.$, $N\left|_{m}\right.$ and $x(m)$. Further, the curvature $\kappa$ and torsion $ \tau$ of $\gamma$ are given by
\begin{subequations}\label{IntCurvEq3ALL}
\begin{eqnarray}
\label{IntCurvEq3a}\kappa&=&\left.\sqrt{1+\frac{1}{4} n^2 H^2}\right|_\gamma,\\
\label{IntCurvEq3b}\tau&=&\left.\frac{ 2n e_1(H)}{4+n^2H^2}\right|_\gamma.
\end{eqnarray}
\end{subequations}
\end{Lemma}
\begin{proof}
Using the notation described above and considering \eqref{BicHyperEq2},  we have
\begin{align}\nonumber
\begin{split}
\quad T'(s)=&-\frac{nH(s)}2N(s)-y(s),\\
\quad N'(s)=&\frac{nH(s)}2N(s),\\
y'(s)=&T(s)
\end{split}
\end{align}
By a direct computation using these equations and \eqref{BicHyperEq1b}, we obtain the usual Frenet-Serret formula 
$$T'=\kappa n,\qquad n'=-\kappa T+c\tau b,\qquad b'=-\tau n $$
with curvature $\kappa$ and  torsion $\tau$ given in \eqref{IntCurvEq3a}, \eqref{IntCurvEq3b}, respectively, where the normal and binormal vector fields $n$ and $b$ are given by
\begin{align}\label{NormBinormSn}
\begin{split}
n(s)=&-\frac{n H(s)}{\sqrt{4 +n^2 H(s)^2}} N(s)-\frac{2 }{\sqrt{4 +n^2 H(s)^2}}y(s),\\
b(s)=& \frac{-2}{\sqrt{4+n^2H(s)^2}} N(s) +  \frac{ n H(s)}{\sqrt{4+n^2H(s)^2}} y(s).
\end{split}
\end{align}
Consequently, $M$ lies on a 3-plane of $\mathbb E^{n+2}$.
\end{proof}

Similarly, we have
\begin{Lemma}\label{Intcurves_e1Lem2Hn}
Let $M$ be a  biconservative hypersurface in $\mathbb H^{n+1}$ with $e_1=\frac{\nabla H}{|\nabla H|}$ and $\gamma$ is an integral curve of $e_1$ passing through $m\in M$ with $H(m)\neq \frac{2}{n}$, where $H$ is the mean curvature of $M$. Then, $\gamma$ lies on a time-like $3$-plane of $\mathbb E^{n+2}_1$ spanned by $e_1\left|_{m}\right.$, $N\left|_{m}\right.$ and $x(m)$. Further, on an open part of $\gamma$ containing $m$, its curvature $\kappa$ and torsion $\tau$ is given by
\begin{subequations}\label{IntCurvEq4ALL}
\begin{eqnarray}
\label{IntCurvEq4a}\kappa&=&\frac{1}{2}\left.\sqrt{| n^2 H^2-4|}\right|_\gamma,\\
\label{IntCurvEq4b}\tau&=&\frac{2 n e_1(H)}{|H^2 n^2-4|}.
\end{eqnarray}
\end{subequations}
\end{Lemma}

\begin{proof}
From equations \eqref{BicHyperEq2} and \eqref{BicHyperEq3a}, we have  $e_A'(s)=0$, $A=2,3,\hdots,n$. Thus, the vector fields $e_2,\hdots,e_n$ are constant along $\gamma$. Also, we have $\langle T(s),e_A(s)\rangle=0$. Hence, $M$ lies on a time-like 3-plane of $\mathbb E^{n+2}$.

Similar to the proof of Lemma \ref{Intcurves_e1Lem2}, we have 
\begin{align}\nonumber
\begin{split}
\quad T'(s)=&-\frac{nH(s)}2N(s)+y(s),\\
\quad N'(s)=&\frac{nH(s)}2N(s),\\
y'(s)=&T(s).
\end{split}
\end{align}
Now, we have $\langle T',T'\rangle=-1 + \frac{1}{4 n^2 H(s)^2}$. Since $\nabla H$ does not vanish by the assumption, the interior of the set $\{m\in\gamma|H(m)=2/n\}|_\gamma$ is empty. Thus, we have either $H(m)>2/n$ or $H(m)<2/n$ on a neighborhood of $m$. In both cases, we have the corresponding Frenet-Serret equations  obtained for curvature  and torsion given by \eqref{IntCurvEq4ALL} with normal and binormal vector fields  
\begin{align}\label{NormBinormHn}
\begin{split}
n(s)=&-\frac{n H(s)}{\sqrt{\varepsilon(-4 +n^2 H(s)^2)}} N(s)+\frac{2 }{\sqrt{\varepsilon(-4 +n^2 H(s)^2)}}y(s),\\
b(s)=& \frac{2\varepsilon}{\sqrt{\varepsilon(-4 +n^2 H(s)^2)}} N(s) -  \frac{ n\varepsilon H(s)}{\sqrt{\varepsilon(-4 +n^2 H(s)^2)}} y(s),
\end{split}
\end{align}
where $\varepsilon=1$ if $n$ is space-like and $\varepsilon=-1$ otherwise.
\end{proof}
Now, we have the following corollary where we consider the distribution
\begin{equation}\label{DefinitionofD}
D=\mathrm{span}\{e_2,e_3,\hdots,e_n\},
\end{equation}
which is integrable because of equation \eqref{BicHyperEq3a}.
\begin{Corol}\label{CorolIntcurves_e1}
Let $M$ be a  biconservative hypersurface in $R^n(c)$ and $e_1=\frac{\nabla H}{|\nabla H|}$, where $H$ is the mean curvature of $M$ and $c\in\{-1,1\}$. Consider an integral submanifold $D$ of $M$ and $m_1,m_2\in D$. Then the integral curves of $e_1$ passing through $m_1$ and $m_2$ are congruent.
\end{Corol}
\begin{proof}
By considering the Lemma \ref{Intcurves_e1Lem2}, Lemma \ref{Intcurves_e1Lem2Hn} and equation \eqref{BicHyperEq1a}, integral curves  of $e_1$ passing through $m_1$ and $m_2$ have same curvature and torsion. Hence, they are congruent.
\end{proof}

Now, we have the following:
\begin{Lemma}\label{LemmaFltNRMBundle}
An integral submanifold of $D$ has flat normal bundle in $R^{n+1}(c)$ as well as in $E(n+1,c)$.
\end{Lemma}

\begin{proof}
Let $\hat M$ be an integral submanifold of $D$. Since $e_2,e_3,\hdots,e_n$ are principal directions of $M$, we have
$\tilde\nabla_{e_A}e_1\in D.$ A direct computation yields
$R^\bot(X,Y)f_n=0$, where $X,Y$ are tangent vector fields to $\hat M$, $R^\bot$ is the normal connection of $\hat M$ in $R^{n+1}(c)$ and $f_n=\left. e_1\right|_{\hat M}$. Hence, $\hat M$ has  flat normal bundle.
\end{proof}
Now, we will give the main result of this subsection which provides a local parametrization for the biconservative hypersurfaces in Riemannian space form.
\begin{theorem}\label{TheoremLocalClassForAnyBicHypr}
Let $M$ be a biconservative hypersurface in $R^{n+1}(c), \ c\in\{-1,1\}$ and $H$ as its mean curvature. Further, assume that $H(m)\neq 2/n$,  $m\in M$ if $c=-1$. Let $\Theta(t_1,t_2,\hdots,t_{n-1})$ be a local parametrization of an integral submanifold $\hat M$ of  the distribution $D$ given by \eqref{DefinitionofD}  passing through $m$. Then, there exists a neighbourhood $\mathcal N_m$ of $m$ on which $M$ can be parametrized as
\begin{align}\label{LocalParametrization}
\begin{split}
x(s,t_1,t_2,\hdots,t_{n-1})=&\Theta(t_1,t_2,\hdots,t_{n-1})+\alpha_1(s)\xi_1(t_1,t_2,\hdots,t_{n-1})+\alpha_2(s)\xi_2(t_1,t_2,\hdots,\\&
t_{n-1})+\alpha_3(s)\xi_3(t_1,t_2,\hdots,t_{n-1})
\end{split}
\end{align}
for any parallel, orthonormal base $\{\xi_1,\xi_2,\xi_3\}$ of the normal space of $\hat M$ in $E(n+1,c)$, where $\alpha_1,\alpha_2,\alpha_3$ are some smooth functions. Furthermore, the slices $t_1=\mbox{constant},\hdots,t_{n-1}=\mbox{constant}$ are integral curves of $e_1$. 
\end{theorem}
\begin{Remark}
Lemma \ref{LemmaFltNRMBundle} yields that $\hat M$ is a flat normal bundle in $E(n+1,c)$. The existence of a parallel base $\{\xi_1,\xi_2,\xi_3\}$
of the normal space of $\hat M$ in $E(n+1,c)$ follows from \cite[Proposition 1.1, p. 99]{ChenThrofSubmn}.
\end{Remark}

\begin{proof}

Let $D$ be the distribution given by \eqref{DefinitionofD} and $D^\bot=\mathrm{span}\{e_1\}$. Since $TM=D\oplus D^\bot$  and $D,D^\bot$ are two integrable submanifolds, there exists a local coordinate system $\{s,t_1,t_2,\hdots,t_{n-1}\}$ on a neighborhood $\mathcal N_m$ of $M$  such that $D^\bot=\mathrm{span}\{\partial_s\} $ and $D=\mathrm{span}\{\partial_{t_1},\partial_{t_2},\hdots,\partial_{t_{n-1}}\}$ (See\cite[Lemma on p. 182]{KobayashiNomizuBook}) where $x(0,0,\hdots,0)=m$. 

We note that equations \eqref{BicHyperEq2} and \eqref{BicHyperEq3ALL} yield that $d\theta_1=0$ where $\theta_1$ is the 1-form defined by $\theta_1(e_i)=\delta_{1i}$. Thus, $\theta_1$ is closed and because of Poincar\`e lemma (by shrinking $\mathcal N_m$ if necessary), we may assume that $\theta_1$ is exact on $\mathcal N_m$. Thus, we may re-define $s$ so that $e_1=\partial_s$. Furthermore, in the case $c=-1$, if necessary, we may shrink $\mathcal N_m$ so that $H(\tilde m)\neq 2/n$ whenever $\tilde m\in \mathcal N_m$. We will obtain a local parametrization of $\mathcal N_m$.

Let $\hat M$ be an integral submanifold of $D$ passing through $m$ and $\Theta(t_1,t_2,\hdots,t_{n-1})$ be its parametrization.  We will consider the cases $c=1$ and $c=-1$ separately. In each cases, we  define two vector fields $n,b$ which are mutually orthonormal and parallel on the normal bundle of $M$ in $E(n+1,c)$. Further, we  put $t_0=e_1{}_m, n_0=n_m, b_0=b_m$, 
$t(t_1,\hdots,t_{n-1})=\left.e_1\right|_{\hat M}$, $n(t_1,\hdots,t_{n-1})=\left.n\right|_{\hat M}$ and $b(t_1,\hdots,t_{n-1})=\left.b\right|_{\hat M}$. 

\textbf{Case 1.} $c=1$. The two vector fields $n,b$ on $M$ are defined by
$$n=-\frac{n H}{\sqrt{4 +n^2 H^2}} N-\frac{2 }{\sqrt{4 +n^2 H^2}}x\quad\mbox{ and }\quad b=
\frac{-2}{\sqrt{4+n^2H^2}} N +  \frac{ n H}{\sqrt{4+n^2H^2}} x.$$

Because of Lemma \ref{Intcurves_e1Lem2Hn}, the integral curve $\gamma$ of  $e_1$ lies on a 3-plane spanned by $t_0,n_0,b_0$. Thus, we have  
$$\gamma(s)=\gamma(0)+\left.\tilde\alpha_1\right|_\gamma t_0+\left.\tilde\alpha_2\right|_\gamma n_0+\left.\tilde\alpha_3\right|_\gamma b_0$$
for some smooth functions $\tilde\alpha_1,\tilde\alpha_2,\tilde\alpha_3$ defined in $M$. Because of {Corollary} \ref{CorolIntcurves_e1}, we have $e_A(\tilde\alpha_a)=0,\ a=1,2,3,\ A=2,3,\hdots,n$. Therefore, we have 
\begin{align}\label{LocalParametrizationBeforeResault}
\begin{split}
x(s,t_1,t_2,\hdots,t_{n-1})=&\Theta(t_1,t_2,\hdots,t_{n-1})+\tilde\alpha_1(s)t(t_1,t_2,\hdots,t_{n-1})+\tilde\alpha_2(s)n(t_1,t_2,\hdots,\\&
t_{n-1})+\tilde\alpha_3(s)b(t_1,t_2,\hdots,t_{n-1}).
\end{split}
\end{align}
Now, for any given parallel, orthonormal base $\{\xi_1,\xi_2,\xi_3\}$ of the normal space of $\hat M$ in $\mathbb E^{n+2}$, we have
\begin{align}\label{LocalParametrizationBeforeResault2}
\begin{split}
\tilde\alpha_1t+\tilde\alpha_2n+\tilde\alpha_3b=\alpha_1\xi_1+\alpha_2\xi_2+\alpha_3\xi_3
\end{split}
\end{align}
for some functions $\alpha_a=\alpha_a(s)$. By combining equations \eqref{LocalParametrizationBeforeResault} and \eqref{LocalParametrizationBeforeResault2}, we obtain  \eqref{LocalParametrization}.

\textbf{Case 2.} $c=-1$. In this case, we define $n$ and $b$ by
$$n=\left\{
\begin{array}{cc}
-\frac{n H}{\sqrt{n^2 H^2-4}} N+\frac{2 }{\sqrt{n^2 H^2-4}}x&\mbox{if $h_*(e_1,e_1)$ is space-like}\\
-\frac{n H}{\sqrt{4 -n^2 H^2}} N+\frac{2 }{\sqrt{4-n^2 H^2}}x&\mbox{if $h_*(e_1,e_1)$ is time-like}
\end{array}
\right.$$
and
$$b=\left\{
\begin{array}{cc}
 \frac{2\varepsilon}{\sqrt{n^2 H^2-4}} N -  \frac{ n\varepsilon H}{\sqrt{n^2 H^2-4}} x&\mbox{if $h_*\nabla_{e_1}e_1$ is space-like}\\
 \frac{2\varepsilon}{\sqrt{4 -n^2 H^2}} N - \frac{ n\varepsilon H}{\sqrt{4 -n^2 H^2}}x&\mbox{if $h_*\nabla_{e_1}e_1$ is time-like}
\end{array}
\right.$$
where $h_*$ is the second fundamental form of $x:M\rightarrow E(n+1,c)$.
Similarly as in Case 1, we have \eqref{LocalParametrization}.
\end{proof}

\subsection{Biconservative hypersurfaces in $R^{n+1}(c)$ with 3 distinct principle curvatures.}

A direct computation yields that if $M$ is a biconservative hypersurface  in the Riemannian space form $R^{n+1}(c)$ with  2 distinct principal curvatures, then it is an open part of a rotational hypersurface in  $R^{n+1}(c)$ for an appropriately chosen profile curve. This can be proved by using a classical result of M. Do Carmo and M. Dajczer (See \cite[Theorem 4.2]{DocarmoDajczer}). It is the reason that we consider biconservative hypersurfaces  with 3 distinct principal curvatures.  

 We would like to give the following lemma which is proved by the exactly same way as done in \cite[Lemma 3.2]{AbhNCTJMAA}.
\begin{Lemma}\label{LemmaKeyR4c}
Let $M$ be a biconservative hypersurface in $R^{n+1}(c)$ with  principal curvatures 
$$k_1=-\frac{nH}2,\qquad k_2=k_3=\cdots=k_{p+1}\neq k_{p+2}=k_{p+3}=\cdots=k_n.$$ 
Then, we have $e_A(k_i)=0$.
\end{Lemma}

\begin{proof}
Due to assumption, equation \eqref{BicEqMod2} becomes
\begin{equation}\label{BicEqKeyLemmaEq01}
3k_1+pk_2+qk_{p+2}=0.
\end{equation}
The Codazzi equation \eqref{CodazziEq1a} implies $\omega_{1a}(e_a)=\omega_1, a=2,3,\hdots,p+1$ and 
$\omega_{1b}(e_b)=\omega_2, b=p+1,p+2,\hdots,n$, where we put $q=n-p-1$. It is to note that if $p>1$ and $q>1$, then the proof follows from the Codazzi equation \eqref{CodazziEq1a}. Therefore, without loss of generality, we assume $p=1$. In this case, equation \eqref{BicEqKeyLemmaEq01} becomes
\begin{equation}\label{BicEqKeyLemmaEq1}
k_2+qk_{p+2}=-3k_1.
\end{equation}
We will prove the lemma for $c=1$. The other case follows from an analogous computation.

Next, we apply $e_1$ to equation \eqref{BicEqKeyLemmaEq1} and use equations \eqref{CodazziEq1a} and \eqref{BicEqKeyLemmaEq1} to get
\begin{equation}\label{BicEqKeyLemmaEq2}
qk_{p+2}  \left(\omega _1-\omega _2\right)+k_1 \left(q \omega _2+4 \omega _1\right)=-3 e_1(k_1).
\end{equation}
By applying $e_1$ twice to equation \eqref{BicEqKeyLemmaEq2} and using equations \eqref{CodazziEq1a}, \eqref{RnGaussEquation} and \eqref{BicEqKeyLemmaEq1}, we obtain
\begin{align}\nonumber
\begin{split}
-2 c qk_{p+2} -8 c k_1+q e_1(k_1) \omega _2+e_1(k_1) \omega _1+\omega _1^2 \left(-2 q k_{p+2} -8 k_1\right)&\\
+q\omega _2^2 \left(2 k_{p+2} -2 k_1 \right)+6 q k_{p+2} k_1^2 +q (q+1)k_{p+2}^2 k_1  -2 qk_1 +2 q k_{p+2} +12 k_1^3&=-3 e_1e_1(k_1)
\end{split}
\end{align}
and
\begin{align}\nonumber
\begin{split}
6 \omega _1^3 \left(q k_{p+2} +4 k_1\right)+6 q \left(k_1-k_{p+2}\right)  \omega _2^3-3 e_1(k_1) \omega _1^2-3q e_1(k_1)  \omega _2^2&\\
+e_1(k_1) \left(-4 (c+q)+6 qk_{p+2} k_1 +qk_{p+2}^2  (q+1)+18 k_1^2\right)+ \omega_1\left(Q(k_1,k_{p+2})+e_1e_1(k_1)\right)&\\
+q \omega _2 \left(e_1e_1(k_1)-10 k_{p+2}-k_1 \left(k_1^2-7 k_{p+2} k_1+6 k_{p+2}^2-10\right)\right)&=-3 e_1e_1e_1(k_1)
\end{split}
\end{align}
for $Q(k_1,k_{p+2})=\left(qk_{p+2} +4 k_1\right) \left(10 c-6q k_{p+2} k_1 -19 k_1^2\right).$

By a direct computation using these equations, we have obtained a non-trivial polynomial equation 
$$\sum\limits_{t=0}^16P_t\left(k_1,e_1(k_1),e_1e_1(k_1),e_1e_1e_1(k_1)\right)k_{p+2}^t=0$$
for some smooth functions $P_0,P_1,\hdots,P_{16},$ such that 
\begin{align}\nonumber
\begin{split}
P_{15}=&88  q^{12} (q+1)^2 (q+4)^3 (q+13) \left(3 (c-1) q (5 q+71)+k_1^2 (q (q (10 q+91)-1105)-484)\right)\\
&k_1^5\left(e_1\left(k_1\right)\right)^2\\ \qquad\mbox{and}\quad
P_{16}=&242  q^{13} (q+1)^3 (q+4)^3 (q+13)^2k_1^6\left(e_1\left(k_1\right)\right)^2.
\end{split}
\end{align}
Since $e_A(k_1)=e_Ae_1(k_1)=e_Ae_1e_1(k_1)=e_Ae_1e_1e_1(k_1)=0$, we have
$$P_t\left(k_1,e_1(k_1),e_1e_1(k_1),e_1e_1e_1(k_1)\right)=\lambda_t,$$
for some constants $\lambda_0,\lambda_1,\hdots,\lambda_{16}$ along an integral curve $\zeta$ of $e_A$. Hence, we have 
$$\lambda_0+\lambda_1k_{p+2}+\hdots\lambda_{16}k_{p+2}^{16}=0,$$
which yields that $k_{p+2}$ is constant along $\zeta$. Hence, we have $e_A(k_i)=0.$
\end{proof}
The next lemma follows from Lemma \ref{LemmaKeyR4c}.
\begin{Lemma}\label{LemmaRnc}
Let $M$ be a biconservative hypersurface in $R^{n+1}(c)$ with principal curvatures  $k_1,k_2,\hdots,k_n$ such that
$$k_1=-\frac{nH}2\neq k_2=k_3=\cdots=k_{p+1}\neq k_{p+2}=k_{p+3}=\cdots=k_n.$$
Define two distributions $D_1,D_2$ on $M$ where $D_1=\{X|SX=k_2X\}$ and $D_2=\{Y|SY=k_{p+2}Y\}.$
Then the Levi-Civita connection of $M$ satisfies 
\begin{align}\nonumber
\begin{split}
\nabla_{e_1}e_1=0,\quad
\nabla_{Y}X\in D_1,\quad
\nabla_{X}Y\in D_2,\qquad\mbox{$\forall\hspace{.1 cm} X\in D_1$ and $Y\in D_2$.}
\end{split}
\end{align}
\end{Lemma}
 
\section{Local classification results in $ R^4(c)$ }
In this section, we give  the complete classification of biconservative hypersurfaces  in $\mathbb S^4$ and $\mathbb H^4$. First, we obtain the following lemma by using Lemma \ref{LemmaRnc}.
\begin{Lemma}\label{BicHyperLeviCvtaR4c}
Let $M$ be a biconservative hypersurface in $R^4(c)$ with three distinct principle curvatures, where $c\in\{-1,+1\}$. Then the Levi-Civita connection of $M$ satisfies
\begin{subequations}\label{BicHyperLeviCvta}
\begin{eqnarray}
\label{BicHyperLeviCvta1}\nabla_{e_1}e_1=0,&\quad \nabla_{e_1}e_2=0,&\quad \nabla_{e_1}e_3=0,\\
\label{BicHyperLeviCvta2}\nabla_{e_2}e_1=\omega_{12}(e_2)e_2,&\quad \nabla_{e_2}e_2=-\omega_{12}(e_2)e_1,&\quad \nabla_{e_2}e_3=0,\\
\label{BicHyperLeviCvta3}\nabla_{e_3}e_1=\omega_{13}(e_3)e_3,&\quad \nabla_{e_3}e_2=0,&\quad \nabla_{e_3}e_3=-\omega_{13}(e_3)e_1.
\end{eqnarray}
\end{subequations}
\end{Lemma}

\subsection{Classification results for $\mathbb S^{4}$}
Let us consider the integrable distribution $D$ given by equation \eqref{DefinitionofD} for $n=3$. Now, we will calculate the integral submanifold of the distribution $D$.
\begin{Prop}\label{Assum1} Any integral submanifold of $D$ is congruent to the flat surface given by
\begin{equation}
\label{BicHyperLeviCvtaS4eq1k}  \Theta(t, u) = \left(c,\frac 1a \cos t, \frac 1a  \sin t,\frac 1b \cos u, \frac 1b \sin u\right)
\end{equation}
for some positive constant $a,b,c$ with $c^2+\frac1{a^2}+\frac1{b^2}=1$.
\end{Prop}

\begin{proof}
Let $\hat M$ be an integral submanifold  of $D$, $j: M\rightarrow\hat M$ the canonical projection. We consider the orthonormal frame field $\{f_1,f_2;f_3,f_4;y\}$ given by
$$f_1=j_*(e_2),\quad\ f_2=j_*(e_3),\qquad f_3=\left.e_1\right|_{\hat M}, \quad f_4=\left.N\right|_{\hat M}, \Theta=x\circ j.$$
Then, equation \eqref{BicHyperLeviCvta} implies
\begin{subequations}\label{BicHyperLeviCvtaS4eq1ALL}
\begin{eqnarray}
\label{BicHyperLeviCvtaS4eq1a}&\quad \widetilde\nabla_{f_1}f_1=-c_1f_3+d_1f_4,&\quad \hat\nabla_{f_1}f_2=0,\\
\label{BicHyperLeviCvtaS4eq1b}&\quad \hat\nabla_{f_2}f_1=0,&\quad \widetilde\nabla_{f_2}f_2=-c_2f_3+d_2f_4.
\end{eqnarray}
Further, we have
\begin{eqnarray}
\label{BicHyperLeviCvtaS4eq1c}\hat\nabla_{f_1}f_3=c_1f_1,&\quad&\hat\nabla_{f_2}f_3=c_2f_2,\\
\label{BicHyperLeviCvtaS4eq1d}\hat\nabla_{f_1}f_4=-d_1f_1,&\quad&\hat\nabla_{f_2}f_3=-d_2f_2,
\end{eqnarray}
where $c_1=\left.\omega_{12}(e_2)\right|_{\hat M}$, $c_2=\left.\omega_{13}(e_3)\right|_{\hat M}$, $d_1=\left.k_2\right|_{\hat M}$, $d_2=\left.k_3\right|_{\hat M}$.
\end{subequations}
It implies that
\begin{eqnarray}
\label{BicHyperLeviCvtaS4eq1e}\hat\nabla_{f_1}f_1=-c_1f_3 + d_1 f_4 -y,\\
\label{BicHyperLeviCvtaS4eq1f}\hat\nabla_{f_2}f_2=-c_2f_3 + d_2 f_4 -y.
\end{eqnarray}
Since $\hat\nabla_{f_1}f_2=0$ and $\hat\nabla_{f_{2}}f_1=0$, we have $\langle R (f_1,f_2)f_1,f_2\rangle =0$, which yields that Gaussian curvature of $\hat M$ is zero. Thus $\hat M$ is flat. Therefore, we have
\begin{equation}
\label{BicHyperLeviCvtaS4eq1g} \Theta(t, u) = \alpha(t) + \beta(u),
\end{equation}
for some smooth vector valued functions $\alpha$ and $\beta$. Now, from equations \eqref{BicHyperLeviCvtaS4eq1e}, \eqref{BicHyperLeviCvtaS4eq1f} and \eqref{BicHyperLeviCvtaS4eq1g}, we have
\begin{eqnarray}
\label{BicHyperLeviCvtaS4eq1h}\alpha^{'''} +a^2\alpha^{'}&=&0,\\
\label{BicHyperLeviCvtaS4eq1i}\beta^{'''}+b^2\beta^{'}&=&0,
\end{eqnarray}
where $a^2=(c_1^2 + d_1^2 + 1)$ and $b^2=(c_2^2 + d_2^2 + 1)$. Further, solving equations \eqref{BicHyperLeviCvtaS4eq1h} and \eqref{BicHyperLeviCvtaS4eq1i} yield that
\begin{equation}
\label{BicHyperLeviCvtaS4eq1j} \Theta(t, u) = C_1 + \cos at C_2+ \sin at C_3 + \cos bu C_4 +  \sin bu C_5
\end{equation}
for some constant vectors $C_1, C_2, C_3, C_4$ and $C_5$, respectively. Therefore, by taking into account that $\{\partial_t,\partial_u\}$ is an orthonormal base and considering $\langle y,y\rangle=1$, we see that, up to rotations, we can assume 
$C_1 = (c, 0, 0, 0, 0)$, $C_2 = (0, \frac 1a , 0, 0, 0)$, $C_3 = (0, 0, \frac 1a , 0, 0)$, $C_4 = (0, 0, 0, \frac 1b, 0)$ and $C_5 = (0, 0, 0, 0, \frac 1b)$ for the constant $c=(1-\frac1{a^2}-\frac1{b^2})^{1/2}$. By re-defining $t,u$ properly, we obtain that $\hat M$ is congruent to the flat surface given by \eqref{BicHyperLeviCvtaS4eq1k}.
\end{proof}
Next, we obtain the following local classifications of biconservative hypersurfaces in $\mathbb S^{4}$.
 
\begin{theorem}\label{Classificationtheoremsphere}
Let $M$ be a hypersurface in $S^{4}(1)$ with diagonalizable shape operator and three distinct principal curvatures. Then, $M$ is biconservative if and only if it is congruent to the submanifolds in $\mathbb E^5$ given by
\begin{equation}\label{BicHyprS4Example}
x(s,t,u)=\left(\alpha_1(s),\alpha_2(s)\cos t,\alpha_2(s)\sin t,\alpha_3(s)\cos t,\alpha_3(s)\sin t\right),
\end{equation}
for a smooth, arc-length parametrized curve $\alpha=(\alpha_1,\alpha_2,\alpha_3):(a,b)\rightarrow \mathbb S^2(1)$  with spherical curvature satisfying
\begin{eqnarray}
\label{ProfileCurvHyprS4ALL}\kappa_S&=&\frac{-3H}2,
\end{eqnarray}
where $H=H(s)$ is the mean curvature of $M$.
\end{theorem}
\begin{proof}
Let $M$ be a biconservative hypersurface in $\mathbb S^4$, $m\in M$ and $\hat M$ be an integral submanifold of the distribution $D$. Then $M$ has a local parametrization given in  Theorem \ref{TheoremLocalClassForAnyBicHypr} and Proposition \ref{Assum1}. It can be assumed that $\hat M$ can have the form given in \eqref{BicHyperLeviCvtaS4eq1k}. We will put $t_1=t$ and $t_2=u$ in this case.

Note that the vector fields 
$\xi_1(t,u)=(1,0,0,0,0)$, $\xi_2(t,u)=(0,\cos t,\sin t,0,0)$ and $\xi_3(t,u)=(0,0,0,\cos u,\sin u)$
form a parallel, orthonormal base for the normal space of $\hat M$ in $\mathbb E^5$. Putting $\xi_1, \xi_2, \xi_3$ in the equation \eqref{LocalParametrization}, we obtain 
\begin{align}\nonumber
\begin{split}
x(s,t,u)=&\left(c+\beta _1(s),\left(\frac{1}{a}+\beta _2(s)\right)\cos  t ,\left(\frac{1}{a}+\beta _2(s)\right)\sin  t ,\left(\frac{1}{b}+\beta _3(s)\right)\cos u ,\right.\\&\left.\left(\frac{1}{b}+\beta _3(s)\right)\sin u \right).
\end{split}
\end{align}
By defining $\alpha_1=c+\beta _1$, $\alpha_2=\frac{1}{a}+\beta _2$ and $\alpha_3=\frac{1}{b}+\beta _3$, we obtain equation \eqref{BicHyprS4Example}. Now, we point out that the integral curve of $e_1$ is congruent to the smooth, arc-length parametrized curve $\alpha=(\alpha_1,\alpha_2,\alpha_3):(a,b)\rightarrow \mathbb S^2(1)$ because of Theorem \ref{TheoremLocalClassForAnyBicHypr}. Thus, Lemma \ref{Intcurves_e1Lem2} yields  that the spherical curvature $\kappa_S$ satisfies equation \eqref{ProfileCurvHyprS4ALL}.
\end{proof}

\subsection{Classification results for $\mathbb H^{4}$}

Similar to previous subsection, first we will obtain integral submanifolds of the distribution $D$ given by equation \eqref{DefinitionofD} for $n=3$.

\begin{Prop}\label{AssumpH41} 
Any integral submanifold of $D$ is congruent to one of the four flat surfaces given below.
\begin{enumerate}
\item A surface given by equation \eqref{BicHyperLeviCvtaS4eq1k}
for some constants $a,b,c$ such that $c^2-\frac1{a^2}-\frac1{b^2}=1$;

\item A surface given by
\begin{equation}
\label{IntSubmnDH42} \Theta(t,u)=\left(\frac{1}{a}\cosh  u,\frac{1}{a}\sinh u,\frac{1}{b} \cos t,\frac{1}{b}\sin t,c\right)
\end{equation}
for some constants $a,b,c$ such that $\frac1{a^2}-c^2-\frac1{b^2}=1$;

\item A surface given by
\begin{equation}
\label{IntSubmnDH43} \Theta(t,u)=\left(a u^2+\frac{a}{b^2}+a+\frac{1}{4 a},u,\frac{1}{b}\cos t,\frac{1}{b}\sin t,a u^2+\frac{a}{b^2}+a-\frac{1}{4 a}\right)
\end{equation}
for some non-zero constants $a,b$;

\item A surface given by
\begin{equation}
\label{IntSubmnDH44} \Theta(t,u)=\left(a \left(t^2+u^2\right)+b,t,u,\sqrt{-\frac{1}{4 a^2}+\frac{b}{a}-1},a \left(t^2+u^2\right)-\frac{1}{2 a}+b\right)
\end{equation}
for some non-zero constants $a,b$.
\end{enumerate}
\end{Prop}
\begin{proof}
Let $\hat M$ be an integral submanifold  of $D$ passing through $m\in M$.
By a similar way in the proof of Proposition \ref{Assum1}, we see that  $\hat M$ is flat and it can be parametrized as $\Theta(t,u)$ given in
equation \eqref{BicHyperLeviCvtaS4eq1g} for some $\mathbb E^5_1$-valued functions $\alpha$ and $\beta$ satisfying
\begin{eqnarray}
\label{BicHyperLeviCvtaH4eq1h}\alpha^{'''} +(c_1^2 + d_1^2 - 1)\alpha^{'}&=&0,\\
\label{BicHyperLeviCvtaH4eq1i}\beta^{'''}+(c_2^2 + d_2^2 - 1)\beta^{'}&=&0,
\end{eqnarray}
where $c_1,c_2,d_1,d_2$ are constants defined in Proposition \ref{Assum1}. Moreover, $\Theta(t,u)$ satisfies
\begin{subequations}\label{BicHyperLeviCvtaH4ThetaEquations}
\begin{eqnarray}
\label{BicHyperLeviCvtaH4ThetaEquations1}\langle\Theta_t,\Theta_t\rangle=\langle\Theta_u,\Theta_u\rangle=1,\quad & 
\langle\Theta_t,\Theta_u\rangle=0,\\
\label{BicHyperLeviCvtaH4ThetaEquations2}\langle\Theta,\Theta\rangle=-1.
\end{eqnarray}
\end{subequations}
Since $\hat M$ is flat, we have $\hat R(\partial_t,\partial_u,\partial_t,\partial_u)=0$, where $\hat R$ is the curvature tensor of $\hat M$. The Gauss equation yields 
$$c_1c_2+d_1d_2=1.$$
Further, an application of well-known Cauchy-Schwarz inequality for the vectors $v=(c_1,d_1)$ and $w=(c_2,d_2)$ yields
$$(c_1^2+d_1^2)(c_2^2+d_2^2)\geq1.$$
Therefore, we have four possible cases:
\begin{itemize}
\item $c_1^2+d_1^2>1,\quad c_2^2+d_2^2>1$,
\item $c_1^2+d_1^2>1,\quad c_2^2+d_2^2<1$,
\item $c_1^2+d_1^2>1,\quad c_2^2+d_2^2=1$,
\item $c_1^2+d_1^2=1,\quad c_2^2+d_2^2=1$.
\end{itemize}

\textbf{Case I.} $c_1^2+d_1^2>1, c_2^2+d_2^2>1$. In this case, by solving equations \eqref{BicHyperLeviCvtaH4eq1h}, \eqref{BicHyperLeviCvtaH4eq1i} and using equation \eqref{BicHyperLeviCvtaS4eq1g}, we obtain $\Theta(t,u)$ as given in equation \eqref{BicHyperLeviCvtaS4eq1j} for some constant vectors $C_1, C_2, C_3, C_4$ and $C_5$. Further, considering equation \eqref{BicHyperLeviCvtaH4ThetaEquations}, we see that $\hat M$ is congruent to the surface given by equation \eqref{BicHyperLeviCvtaS4eq1k}.

\textbf{Case II.} $c_1^2+d_1^2>1, c_2^2+d_2^2<1$. In this case, by solving equations \eqref{BicHyperLeviCvtaH4eq1h}, \eqref{BicHyperLeviCvtaH4eq1i} and using equation \eqref{BicHyperLeviCvtaS4eq1g}, we obtain $\Theta(t,u)$ as given in 
\begin{equation}
\nonumber \Theta(t, u) = C_1 + \cos at C_2+ \sin at C_3 + \cosh bu C_4 +  \sinh bu C_5
\end{equation}
for some constant vectors $C_1, C_2, C_3, C_4$ and $C_5$. Again considering equation \eqref{BicHyperLeviCvtaH4ThetaEquations}, we see that $\hat M$ is congruent to the surface given by equation \eqref{IntSubmnDH42}.

\textbf{Case III.} $c_1^2+d_1^2>1, c_2^2+d_2^2=1$. In this case, by solving equations \eqref{BicHyperLeviCvtaH4eq1h}, \eqref{BicHyperLeviCvtaH4eq1i} and using equation \eqref{BicHyperLeviCvtaS4eq1g}, we obtain $\Theta(t,u)$ as given in 
\begin{equation}
\nonumber \Theta(t, u) = C_1 + \cos bt C_2+ \sin bt C_3 + u^2 C_4 +  u C_5
\end{equation}
for a non-zero constant $b$ and some constant vectors $C_1, C_2, C_3, C_4$ and $C_5$. Considering equation \eqref{BicHyperLeviCvtaH4ThetaEquations1}, we obtain 
$\langle C_2,C_2\rangle=\langle C_3,C_3\rangle=1/b^2$, $\langle C_4,C_4\rangle=0$, $\langle C_5,C_5\rangle=1$
and
$\langle C_a,C_b\rangle=0$ if $a\neq b$,  $a,b>1$. Thus, up to congruency we may assume
$C_2=(0,0,1/b,0,0), $ $C_3=(0,0,1/b,0,0), $, $C_4=(a,0,0,0,a), $ and $C_2=(0,1,0,0,0) $ for a constant $a\neq0$. Finally, by considering equation \eqref{BicHyperLeviCvtaH4ThetaEquations2}, we conclude that 
$\hat M$ is congruent to the surface given by equation \eqref{IntSubmnDH43}.

\textbf{Case IV.} $c_1^2+d_1^2=1, c_2^2+d_2^2=1$. In this case, by solving equations \eqref{BicHyperLeviCvtaH4eq1h}, \eqref{BicHyperLeviCvtaH4eq1i} and using equation \eqref{BicHyperLeviCvtaS4eq1g}, we obtain $\Theta(t,u)$ as given in 
\begin{equation}
\nonumber \Theta(t, u) = C_1 + t^2 C_2+ t C_3 + u^2 C_4 +  u C_5
\end{equation}
for a non-zero constant $b$ and some constant vectors $C_1, C_2, C_3, C_4$ and $C_5$. By the same way in the Case III, we obtain that $\hat M$ is congruent to the surface  given by equation \eqref{IntSubmnDH44}.
\end{proof}

\begin{theorem}\label{Classificationtheoremhypr}
A biconservative hypersurface $M$ in $\mathbb H^{4}$ with three distinct principal curvatures is congruent to one of the four hypersurfaces given below.
\begin{enumerate}
\item A hypersurface in $\mathbb H^4$ given by
equation \eqref{BicHyprS4Example}
for a smooth, arc-length parametrized curve $\alpha=(\alpha_1,\alpha_2,\alpha_3):(a,b)\rightarrow \mathbb H^2(-1)$;
\item
A hypersurface in $\mathbb H^4$ given by
\begin{equation}\label{BicHyprH4Example2}
x(s,t,u)=\left(\alpha_1(s)\cosh t,\alpha_1(s)\sinh t,\alpha_2(s)\cos t,\alpha_2(s)\sin t,\alpha_3(s)\right)
\end{equation} 
for a smooth, arc-length parametrized curve $\alpha=(\alpha_1,\alpha_2,\alpha_3):(a,b)\rightarrow \mathbb H^2(-1)$;
\item
A hypersurface in $\mathbb H^4$ given by
\begin{align}\label{BicHyprH4Example3}
\begin{split}
x(s,t,u)=&\left(\frac{a A(s)^2+a}{s}+a s u^2+\frac{s}{4 a},s u,A(s) \cos t,A(s) \sin t,\right.\\&\left.
\frac{a A(s)^2+a}{s}+a s u^2-\frac{s}{4 a}\right)
\end{split} 
\end{align} 
for smooth functions $\alpha_2,\alpha_3$ and some non-zero constants $a,\alpha_1$;

\item
A hypersurface in $\mathbb H^4$ given by
\begin{align}\label{BicHyprH4Example4}
\begin{split}
x(s,t,u)=&\left(\frac{a A(s)^2}{s}+a s \left(t^2+u^2\right)+\frac{s}{4 a}+\frac{a}{s},s t,s u,A(s),\right.\\&\left.
\frac{a A(s)^2}{s}+a s \left(t^2+u^2\right)-\frac{s}{4 a}+\frac{a}{s}\right)
\end{split} 
\end{align} 
for a smooth function $A$ and a non-zero constant $a$;
\end{enumerate}
\end{theorem}
\begin{proof}
Let $M$ be a biconservative hypersurface in $\mathbb H^{4}$ with three distinct principal curvatures and $D$ is the distribution given by equation \eqref{DefinitionofD} for $n=3$, $m\in M$. Suppose $\Theta(t,u)$ be a parametrization of integral submanifolds $\hat M$ of $D$ passing through $m$. Then, it is in one of four forms given in Proposition \ref{AssumpH41}. Therefore, we have four cases.

\textbf{Case 1 and Case 2}. Let $\Theta$ has the form either given in equation \eqref{BicHyperLeviCvtaS4eq1k} or equation \eqref{IntSubmnDH42}.

In this case, by similar computations that we did in the proof of Theorem \ref{Classificationtheoremsphere}, we obtain that $M$ is congruent to one of hypersurfaces given in Case 1 and Case 2 of the theorem. 

\textbf{Case 3}. Suppose $\Theta$ has the form given in equation \eqref{IntSubmnDH43}. Then, the normal vector fields 
$$\xi_1(t,u)=\left(0,0,\cos  t,\sin  t,0\right),\hspace{.5 cm} \xi_2(t,u)=\left(\frac{4 a^2 u^2+3}{2 \sqrt{2}},\sqrt{2} a u,0,0,\frac{4 a^2 u^2+1}{2 \sqrt{2}}\right)$$ and $$\xi_3(t,u)=\left(\frac{1-4 a^2 u^2}{2 \sqrt{2}},-\sqrt{2} a u,0,0,\frac{3-4 a^2 u^2}{2 \sqrt{2}}\right)$$
form an parallel, orthonormal base for the normal space of $\hat M$ in $\mathbb E^5_1$. Combining these equations with equation \eqref{LocalParametrization}, we obtain 
\begin{align}\nonumber
\begin{split}
x(s,t,u)=&\left(u^2 \left(\sqrt{2} a^2 \beta _2(s)-\sqrt{2} a^2 \beta _3(s)+a\right)+\frac{a}{b^2}+a+\frac{1}{4 a}+\frac{3 \beta _2(s)}{2 \sqrt{2}}+\frac{\beta _3(s)}{2 \sqrt{2}},\right.\\&\left.
u \left(\sqrt{2} a \beta _2(s)-\sqrt{2} a \beta _3(s)+1\right),\cos  t \left(\frac{1}{b}+\beta _1(s)\right), \sin t \left(\frac{1}{b}+\beta _1(s)\right),\right.\\&\left.
u^2 \left(\sqrt{2} a^2 \beta _2(s)-\sqrt{2} a^2 \beta _3(s)+a\right)+\frac{a}{b^2}+a-\frac{1}{4 a}+\frac{\beta _2(s)}{2 \sqrt{2}}+\frac{3 \beta _3(s)}{2 \sqrt{2}} \right).
\end{split}
\end{align}
By defining $\alpha_1=\sqrt{2} a \beta _2(s)-\sqrt{2} a \beta _3(s)+1$, $\alpha_2=\frac{a}{b^2}+a+\frac{3 \beta _2(s)}{2 \sqrt{2}}+\frac{\beta _3(s)}{2 \sqrt{2}}$ and $\alpha_3=\frac{1}{b}+\beta _1(s)$, we obtain 
\begin{align}\label{IntSubmnDH43ParametrizationEq1}
\begin{split}
x(s,t,u)=&\left(a u^2 \alpha _1(s)+\frac{1}{4 a}+\alpha _2(s),u \alpha _1(s),\alpha _3(s) \cos  t,\alpha _3(s) \sin  t,\right.\\&\left.
\frac{1-2 \alpha _1(s)}{4 a}+a u^2 \alpha _1(s)+\alpha _2(s)\right).
\end{split}
\end{align}
Now, we want to prove the following assumption.
\begin{Assu}\label{MyAssu01}
If $\alpha_1=c$, then the hypersurface given by equation \eqref{IntSubmnDH43ParametrizationEq1} has constant mean curvature.
\end{Assu}
\textit{Proof of Assumption \ref{MyAssu01}.} If $\alpha_1=c$ is constant, then equation \eqref{IntSubmnDH43ParametrizationEq1} becomes 
\begin{align}\label{IntSubmnDH43ParametrizationEq2}
\begin{split}
x(s,t,u)=&\left(a c u^2+\frac{1}{4 a}+\alpha _2(s),c u,\alpha _3(s) \cos  t,\alpha _3(s) \sin  t,\right.\\&\left.
a c u^2+\frac{1-2 c}{4 a}+\alpha _2(s)\right).
\end{split}
\end{align}
Further, considering $\langle x,x\rangle=-1$ yields that 
$$\frac{c \left(-4 a \alpha _2(s)+c-1\right)}{4 a^2}+\alpha _3(s){}^2=-1.$$
Thus, we have $\alpha _2(s)= \frac{4 a^2 \alpha _3(s){}^2+4 a^2+c^2-c}{4 a c}$. Therefore, equation \eqref{IntSubmnDH43ParametrizationEq2} becomes 
\begin{align}\label{IntSubmnDH43ParametrizationEq3}
\begin{split}
x(s,t,u)=&\left(\frac{a \alpha _3(s){}^2}{c}+a \left(c u^2+\frac{1}{c}\right)+\frac{c}{4 a},c u,\alpha _3(s) \cos  t,\alpha _3(s) \sin  t,\right.\\&\left.
\frac{a \alpha _3(s){}^2}{c}+a \left(c u^2+\frac{1}{c}\right)-\frac{c}{4 a}\right).
\end{split}
\end{align}
However, a direct computation yields that the shape operator $S$ of equation \eqref{IntSubmnDH43ParametrizationEq3} is the identity operator acting on $TM$. This proves the Asumption \ref{MyAssu01}.\QEDA

Since $\alpha'$ is not a zero function, we may define a new local coordinate function $U$ by $U=\alpha_1(s)$ and two other functions by $B(s)=\alpha_2(\alpha_1^{-1}(U)),\ A(s)= \alpha_2(\alpha_1^{-1}(U))$. Considering $\langle x,x\rangle=-1$, we obtain 
$$B(U)=\frac{4 a^2 A(U)^2+4 a^2+U^2-U}{4 a U}.$$
Therefore, combining this definition and replacing $U$ by $s$, we obtain equation \eqref{BicHyprH4Example3}. It is important to note that the induced metric of $M$ is 
\begin{equation}\label{BicHyprH4Example3Metric}
g=\frac{\left(A(s)-s A'(s)\right)^2+1}{s^2}ds^2+A(s)^2dt^2+s^2du^2.
\end{equation}
Thus, $A$ is non-vanishing.

\textbf{Case 4}. Let $\Theta$ has the form given in equation \eqref{IntSubmnDH44}. Then, the normal vector fields 
$$\xi_1(t,u)=\left(0, 0, 0, 1, 0\right), \hspace{.5 cm}\xi_2(t,u)=\left(\frac{4 a^2 \left(t^2+u^2\right)+3}{2 \sqrt{2}},\sqrt{2} a t,\sqrt{2} a u,0,\frac{4 a^2 \left(t^2+u^2\right)+1}{2 \sqrt{2}}\right)$$ and 
$$\xi_3(t,u)=\left(\frac{1-4 a^2 \left(t^2+u^2\right)}{2 \sqrt{2}},-\sqrt{2} a t,-\sqrt{2} a u,0,\frac{3-4 a^2 \left(t^2+u^2\right)}{2 \sqrt{2}}\right)$$
form an parallel, orthonormal base for the normal space of $\hat M$ in $\mathbb E^5_1$. 

Therefore,  Theorem \ref{TheoremLocalClassForAnyBicHypr} implies that
\begin{align}\nonumber
\begin{split}
x(s,t,u)=&\left(a \alpha _1(s) \left(t^2+u^2\right)+\alpha _2(s),t \alpha _1(s),u \alpha _1(s),\alpha _3(s),-\frac{\alpha _1(s)}{2 a}+a \alpha _1(s) \left(t^2+u^2\right)+\alpha _2(s)\right),
\end{split}
\end{align}
where $\alpha_1=\sqrt{2} a \beta _2-\sqrt{2} a \beta _3+1,$ $\alpha_2=b+\frac{3 \beta _2}{2 \sqrt{2}}+\frac{\beta _3}{2 \sqrt{2}}$ and $\alpha_3=\sqrt{-\frac{1}{4 a^2}+\frac{b}{a}-1}+\beta _1.$ Again, by the similar calculation as in Case 3, we obtain equation \eqref{BicHyprH4Example4}.
\end{proof}

In the remaining part of this section, we emphasis to show existence of biconservative surfaces with non-constant mean curvature belonging to hypersurface family given by equations \eqref{BicHyprH4Example3} and \eqref{BicHyprH4Example4}.

Let $M$ be a hypersurface in $\mathbb H^4$ given by equation \eqref{BicHyprH4Example3} for a smooth non-vanishing function $A$. Since the induced metric $g$ of $M$ has the form given by equation \eqref{BicHyprH4Example3Metric}, therefore, $\displaystyle e_1=\frac{s}{\sqrt{\left(A(s)-s A'(s)\right)^2+1}} \partial_s$, $\displaystyle e_2=\frac 1{A(s)}\partial_t$ and $e_3=\frac 1s\partial_u$ form an orthonormal frame field for the tangent bundle of $M$. Furthermore, the unit normal vector field of $M$ in $\mathbb H^4$ is given by
\begin{align}\nonumber
\begin{split}
N=&\frac{s}{\sqrt{\left(A(s)-s A'(s)\right)^2+1}}\left({ \frac{2 a A(s) A'(s)}{s}-\frac{a \left(A(s)^2+1\right)}{s^2}+a u^2+\frac{1}{4 a}},{ u},{ \cos (t) A'(s)},\right.\\&\left.
{ \sin (t) A'(s)},\frac{2 a A(s) A'(s)}{s}-\frac{a \left(A(s)^2+1\right)}{s^2}+a u^2-\frac{1}{4 a}\right).
\end{split}
\end{align}
By  direct computations, we obtain that $e_1,\ e_2$ and $e_3$ are principal directions of $M$ with corresponding principal curvatures given by 
\begin{align}\label{BicHyprH4Example3PrnCrvts}
\begin{split}
k_1=&\frac{A''(s) s^2+A'(s) \left(\left(A(s)-s A'(s)\right)^2+1\right) s-A(s) \left(\left(A(s)-s A'(s)\right)^2+1\right)}{\left(\left(A(s)-s A'(s)\right)^2+1\right)^{3/2}},\\
k_2=&\frac{-A(s)^2+s A'(s) A(s)-1}{A(s) \sqrt{\left(A(s)-s A'(s)\right)^2+1}},\\
k_3=& \frac{s A'(s)-A(s)}{\sqrt{\left(A(s)-s A'(s)\right)^2+1}}.
\end{split}
\end{align}
Therefore, $e_1$ is proportional to $\nabla H$ which yields that $M$ is biconservative if and only if \linebreak $3k_1+k_2+k_3=0$. Combining this equation with equation \eqref{BicHyprH4Example3PrnCrvts}, we obtain the following.
\begin{Prop}
Let $M$ be the hypersurface in $\mathbb H^4$ given by equation \eqref{BicHyprH4Example3} for a non-vanishing function $A$ with non-constant mean curvature. Then,  $M$ is biconservative if and only if $A$ satisfies the second order ODE
$$3 s^2 A A''+5 s^3 A A'^3+\left(-15 s^2 A^2-s^2\right) A'^2+\left(15 s A^3+7 s A\right) A'-5 A^4-6 A^2-1=0.$$
\end{Prop}

Next, we assume that $M$ is the hypersurface given by equation \eqref{BicHyprH4Example3}. By a direct computation, we see that $\partial_s,\ \partial_t,\ \partial_u$ are principal directions of $M$ with corresponding principal curvatures given by 
\begin{align}\nonumber
\begin{split}
k_1=&\frac{A''(s) s^2+A'(s) \left(\left(A(s)-s A'(s)\right)^2+1\right) s-A(s) \left(\left(A(s)-s A'(s)\right)^2+1\right)}{\left(\left(A(s)-s A'(s)\right)^2+1\right)^{3/2}},\\
k_2=k_3=& \frac{s A'(s)-A(s)}{\sqrt{\left(A(s)-s A'(s)\right)^2+1}}.
\end{split}
\end{align}
Hence, we have the following.
\begin{Prop}
Let $M$ be the hypersurface in $\mathbb H^4$ given by equation \eqref{BicHyprH4Example4} for a non-vanishing function $A$ with non-constant mean curvature. Then,  $M$ is biconservative if and only if $A$ satisfies the second order ODE
$$3 s^2 A''+5 s^3 A'^3-15 s^2 A A'^2+\left(15 s A^2+5 s\right) A'-5 A^3-5 A=0.$$
\end{Prop}

\end{document}